\documentclass[12pt]{article}
\usepackage[dvips]{graphicx}
\usepackage{amssymb,color,latexsym}
\usepackage[centertags]{amsmath}
\usepackage{amsfonts}
\usepackage{comment}
\usepackage{amsmath}
\usepackage{enumerate}

\def\BBox{\kern  -0.2cm\hbox{\vrule width 0.2cm height 0.2cm}}

{\makeatletter
 \gdef\xxxmark{%
   \expandafter\ifx\csname @mpargs\endcsname\relax 
     \expandafter\ifx\csname @captype\endcsname\relax 
       \marginpar{{xxx}}
     \else
       {xxx} 
     \fi
   \else
     {xxx} 
   \fi}
 \gdef\xxx{\@ifnextchar[\xxx@lab\xxx@nolab}
 \long\gdef\xxx@lab[#1]#2{{\bf [\xxxmark #2 ---{\sc #1}]}}
 \long\gdef\xxx@nolab#1{{\bf  [\xxxmark #1]}}}

\newtheorem{theorem}{Theorem}
\newtheorem{lemma}[theorem]{Lemma}

\newtheorem{corollary}[theorem]{Corollary}
\newtheorem{proposition}[theorem]{Proposition}

\newcommand{\Po}{\mathcal{P}}
\newcommand{\K}{\mathcal{K}}

\newcommand{\s}{\sigma}
\newcommand{\Q}{\mathcal{Q}}
\newcommand{\calP}{\mathcal P}
\newcommand{\calK}{\mathcal K}
\newcommand{\G}{\Gamma}
\newcommand{\calR}{\mathcal R}
\newcommand{\ch}[1]{\overline{#1}}
\newcommand{\eps}{\varepsilon}

\title{Chiral extensions of chiral polytopes}
\author{Gabe Cunningham\thanks{gabriel.cunningham@gmail.com
617-651-2358} \and Daniel Pellicer\thanks{pellicer@matmor.unam.mx}}

\begin{document}
\maketitle

\begin{abstract}
Given a chiral $d$-polytope $\mathcal K$ with regular facets, we describe a construction for a chiral
$(d+1)$-polytope $\Po$ with facets isomorphic to $\mathcal K$. Furthermore, $\Po$ is finite whenever $\K$ is finite.
We provide explicit examples of chiral 4-polytopes constructed in this way from chiral toroidal maps.
\end{abstract}

\section{Introduction}
\label{SectionIntroduction}
Abstract polytopes are combinatorial structures that mimic convex polytopes in several key ways.
They also generalize (non-degenerate) maps on surfaces and face-to-face tessellations of euclidean,
hyperbolic, and projective spaces. Regular polytopes have full symmetry by (abstract) reflections and
have been extensively studied \cite{ARP}. One of the most important problems in the study of regular polytopes
is the \emph{extension problem}: given a regular polytope $\calK$ of rank $d$, what sorts of regular polytopes
of rank $d+1$ have facets isomorphic to $\calK$? Though this problem is far from solved, many useful
partial results already exist (see, for example, \cite{dan} and \cite{arr}). In particular, if $\calK$ is a finite regular polytope, then
\cite{pre2} shows how to construct infinitely many finite regular polytopes with facets isomorphic to $\calK$.

Another important class of polytopes are the \emph{chiral polytopes}, which have full rotational symmetry,
but no symmetry by reflection. There are many examples in ranks 3 and 4 (see \cite{chiral-mix, chiregas2}
for some of these). In higher ranks, however, we have only a handful of concrete examples. 

Many of the important unsolved problems of chiral polytopes are summarized in \cite{chiral-problems}.
Problems 24-30 all concern the extension problem for chiral polytopes, signifying both the importance
of that general problem and how little is known. An important partial result was given in
\cite{higherranks}, where it is shown how to build a finite chiral polytope of rank $d+1$ with
facets isomorphic to a finite regular polytope $\calK$ of rank $d$. There are very restrictive conditions on
the polytope $\calK$, however, so more work remains to be done even on this piece of the extension
problem (Problem 27 of \cite{chiral-problems}).

In this paper we use GPR graphs (as defined in \cite{pewe}) to build chiral polytopes of rank $d+1$ 
with facets isomorphic to a given chiral polytope of rank $d$. In particular, Theorem~\ref{MainProposition}
implies the following:
\begin{theorem}\label{MainTheorem}
Every finite chiral $d$-polytope with regular facets is itself the facet of a finite chiral $(d+1)$-polytope.
\end{theorem}
This gives a partial answer to Problem 26 in \cite{chiral-problems}. We note that the assumption that
the chiral $d$-polytope has regular facets is necessary (see \cite[Proposition 9]{chiregas}). 

We will start by giving background on polytopes in Section~\ref{SectionDefinitions} and on GPR graphs in Section~\ref{SectionGraphs}.
Section~\ref{SectionResult} details the main construction, culminating in Theorem~\ref{MainProposition}.
Finally, in Section~\ref{SectionExamples}, we will apply the construction to a family of chiral toroidal maps
$\{4, 4\}_{(b, c)}$ and analyze the structure of the resulting chiral polytope of rank 4.

\section{Regular and chiral polytopes}
\label{SectionDefinitions}

   In this section we introduce abstract regular and chiral polytopes,
referring to \cite{ARP} and \cite{chiregas} for details.

   An {\em (abstract) $d$-polytope} $\mathcal K$
is a partially ordered set whose elements are called {\em faces} and which satisfies the
following properties. It contains a minimum face $F_{-1}$ and a maximum face $F_d$, and every
{\em flag} of ${\mathcal K}$ (maximal totally ordered subset) contains precisely
$d+2$ elements, including $F_{-1}$ and $F_d$. This induces a rank function from
$\mathcal K$ to the set $\{-1, 0, \dots, d\}$ such that $rank(F_{-1})= -1$ and
$rank(F_d)=d$. The faces of rank $i$ are called {\em $i$-faces}, the $0$-faces are called
{\em vertices}, the $1$-faces are called {\em edges} and the $(d-1)$-faces are
called {\em facets}. Furthermore, we say that $\K$ has rank $d$. In analogy with convex
polytopes, an abstract 3-polytope is also called a \emph{polyhedron}. We shall abuse notation
and identify the {\em section} $G/F_{-1} := \{H \,|\, H\le G\}$ with the face $G$ itself. Given a vertex $v$, the section
$F_d/v := \{H \,|\, H\ge v\}$ is called the {\em vertex-figure} of $\mathcal K$ at $v$.
For every pair of incident faces $F \leq G$ such that $rank(G) - rank(F) = 2$, there
exist precisely two faces $H_1$ and $H_2$ such that $F < H_1, H_2 < G$. This property
is referred to as the {\em diamond condition}. As a consequence of the diamond condition,
for any flag $\Phi$ and any $i \in \{0, \dots, d-1\}$ there exists a unique flag
$\Phi^i$ that differs from $\Phi$ only in the $i$-face. This flag is called the {\em $i$-adjacent flag} of $\Phi$. Finally,
$\mathcal K$ must be {\em strongly flag-connected}, meaning that for any two flags
$\Phi$, $\Phi'$ there exists a sequence of flags
$\Phi = \Psi_0, \Psi_1, \dots, \Psi_m = \Phi'$
such that $\Phi \cap \Phi' \subseteq \Psi_k$, and $\Psi_{k-1}$ is adjacent
to $\Psi_k$ for $k = 1, \dots, m$.

	If $F$ is an $(i-2)$-face and $G$ is an $(i+1)$-face of the $d$-polytope $\calK$, with
$F < G$, then the section $G/F := \{H \mid F \leq H \leq G\}$ is an abstract polygon.
If $\calK$ has the property that the type of each of these sections depends only on $i$
(and not on the particular choice of $F$ and $G$), then we say that $\calK$ is
\emph{equivelar}. In this case, $\calK$ has a {\em Schl\"afli type} (or \emph{Schl\"afli symbol})
$\{p_1, \ldots, p_{d-1}\}$, where the section $G/F$ is a $p_i$-gon whenever $F$ is an $(i-2)$-face and
$G$ is an $(i+1)$-face with $F < G$. The numbers $p_i$ satisfy $2 \leq p_i \leq \infty$, but
in this paper we will always have $3 \leq p_i$. Regular and chiral polytopes, defined
below, are examples of equivelar polytopes.

   An {\em automorphism} of a $d$-polytope $\mathcal K$ is an order-preserving permutation
of its faces. The group of automorphisms of $\K$ is denoted by
$\Gamma({\mathcal K})$. There is a natural action of $\G(\calK)$ on the flags of
$\calK$, and we say that $\mathcal K$ is {\em regular} if this action is transitive.
In this case, $\Gamma({\mathcal K})$ is
generated by involutions $\rho_0, \dots, \rho_{d-1}$, where $\rho_i$ is the unique automorphism mapping a fixed {\em base flag} $\Phi$ to its
$i$-adjacent flag $\Phi^i$. These generators satisfy the relations
\begin{eqnarray*}
\rho_i^2 &=& \varepsilon, \\
(\rho_i \rho_j)^2 &=& \varepsilon \quad \mbox{whenever $|i-j| \ge 2$,}
\end{eqnarray*}
where $\varepsilon$ denotes the identity element.
Regular polytopes are equivelar, and the order of the element $\rho_{i-1} \rho_i$
coincides with the Schl\"afli number $p_i$.

The generators $\{\rho_0, \ldots, \rho_{d-1}\}$ also satisfy the intersection
conditions given by
\begin{equation}\label{refinter}
\langle \rho_i \,|\, i \in I \rangle \cap \langle \rho_i \,|\, i \in J \rangle = \langle
\rho_i \,|\, i \in I \cap J\rangle,
\end{equation}
for all $I, J \subseteq \{0, \dots, d-1\}$.

A {\em string C-group} is a group together with a generating set $\{\rho_0, \dots, \rho_{d-1}\}$ such that the generators $\rho_i$
are involutions satisfying the relation $(\rho_i \rho_j)^2 = \varepsilon$ for $|i-j| \ge 2$, and the intersection condition
(\ref{refinter}). The string C-groups are in a one-to-one
correspondence with the automorphism groups of regular polytopes; in particular, every regular
polytope can be reconstructed from its automorphism group.

   The {\em rotation subgroup} of (the automorphism group of) a regular $d$-polytope
$\mathcal K$ is defined as the subgroup $\Gamma^+({\mathcal K})$ of $\Gamma({\mathcal K})$ consisting of all elements that can be
expressed as words of even length on the generators $\rho_0, \dots, \rho_{d-1}$. The index of $\Gamma^+({\mathcal K})$ in
$\Gamma({\mathcal K})$ is at most $2$. Whenever $\Gamma^+({\mathcal K})$ has index $2$
in $\Gamma({\mathcal K})$ we say that $\mathcal K$ is {\em orientably regular}; other sources also use
the term \emph{directly regular} (see, for example, \cite{chiregas}).

For $i=1, \dots, d-1$ we define the
{\em abstract rotation} $\sigma_i$ to be $\rho_{i-1} \rho_i$, that is, the
automorphism of $\mathcal K$ mapping the base flag $\Phi$ to $(\Phi^i)^{i-1}$.
Then $\Gamma^+({\mathcal K}) = \langle \sigma_1, \dots, \sigma_{d-1} \rangle$
and the abstract rotations satisfy the relations
\begin{equation}\label{sigmarelation}
(\sigma_i \cdots \sigma_j)^2 = \varepsilon
\end{equation}
for $i < j$. The order of $\sigma_i$ is just the entry $p_i$ in the
Schl\"afli symbol.

   We define {\em abstract half-turns} as the involutions
$\tau_{i, j} := \sigma_i \cdots \sigma_j$ for $i<j$. For consistency we also
define $\tau_{0, i} := \tau_{i, d} := \varepsilon$ and
denote $\sigma_i$ by $\tau_{i, i}$.
Then the abstract rotations and
half-turns satisfy the intersection condition given by
\begin{equation}\label{intersectprop}
\langle \tau_{i, j} \,|\, i \le j; i-1, j \in I \rangle \cap
\langle \tau_{i, j} \,|\, i \le j; i-1, j \in J \rangle =
\langle \tau_{i, j} \,|\, i \le j; i-1, j \in I \cap J \rangle
\end{equation}
for $I, J \subseteq \{-1, \dots, d\}$.

   We say that the $d$-polytope $\mathcal K$ is {\em chiral} if its automorphism group
induces two orbits on flags with the property that adjacent flags always belong to
different orbits.  The facets and vertex-figures of a chiral polytope must be either orientably regular or chiral, and the
$(d-2)$-faces must be orientably regular (see \cite[Proposition 9]{chiregas}).

   The automorphism group $\Gamma({\mathcal K})$ of a chiral polytope is generated
by elements $\sigma_1, \dots, \sigma_{d-1}$, where
$\sigma_i$ maps a base flag $\Phi$ to $(\Phi^i)^{i-1}$. That is, $\sigma_i$
cyclically permutes the $i$- and $(i-1)$-faces of $\mathcal K$ incident
with the $(i-2)$- and $(i+1)$-faces of $\Phi$. Furthermore, the generators
$\sigma_i$ also satisfy (\ref{sigmarelation}) as well as the intersection conditions
(\ref{intersectprop}). Because of the obvious similarities between the automorphism
group of a chiral polytope and the rotation subgroup of a regular polytope we shall also
refer to the generators $\sigma_i$ of the automorphism group of a chiral polytope as
{\em abstract rotations}, to the products $\tau_{i, j} :=
\sigma_i \sigma_{i+1} \cdots \sigma_j$ ($i < j$)
as {\em abstract half-turns}, and to the automorphism group of a chiral polytope
as its {\em rotation subgroup}. When convenient, we shall denote the automorphism group
$\Gamma(\K) = \langle \sigma_1, \dots, \sigma_{d-1} \rangle$ of the chiral polytope
$\K$ by $\Gamma^+(\K)$.

   Each chiral $d$-polytope $\K$ occurs in two {\em enantiomorphic forms}, which are, in a sense, a right-
and left-handed version which can be thought of as mirror images of each other. When we fix a base flag
of a chiral polytope, we are essentially choosing an orientation, and the presentation for the automorphism
group depends on which flag we pick. If $\calK$ is a chiral polytope (with a chosen base flag $\Phi$), then we denote
the remaining enantiomorphic form by $\ch{\calK}$, which corresponds to choosing $\Phi^{d-1}$ as the base flag instead. In fact, since all flags adjacent to $\Phi$ are in the same orbit under $\Gamma(\K)$, we may choose the base flag of $\ch{\K}$ to be $\Phi^j$ for any $j \in \{0, \dots, d-1\}$.  By picking $j < d-1$, it is clear that if $\Po$ is an enantiomorphic form of a chiral polytope with chiral facets isomorphic to $\K$, then the facets of $\ch{\Po}$ are isomorphic to $\ch{\K}$.
For further details, see \cite{chiregas2}.

	If $w$ is a word on the generators $\{\s_1, \ldots, \s_{d-1}\}$ then we define the \emph{enantiomorphic word}
$\ch{w}$ to be the word obtained from $w$ by changing every factor $\s_{d-1}$ to $\s_{d-1}^{-1}$ and $\s_{d-2}$ to $\s_{d-2} \s_{d-1}^2$.
Note that if $\K$ is regular then $\ch{w} = \rho_{d-1} w \rho_{d-1}$.
In any case, $w = \eps$ in $\G(\calK)$ if and only if $\ch{w} = \eps$ in $\G(\ch{\calK})$ (see \cite{chiral-mix}).

	We can measure the degree of chirality of a polytope by using its \emph{chirality group} (see \cite{chiral-mix, Gabe}).
Given a chiral $d$-polytope $\calK$, we consider the largest group $\Delta$ with generators $\mu_1, \dots, \mu_{d-1}$ that is covered by both $\G(\calK)$ and
$\G(\ch{\calK})$ by mapping the $i$-th abstract rotation to $\mu_i$. We can obtain $\Delta$ by adding the relations from $\G(\ch{\calK})$ to those of $\G(\calK)$, rewriting the
relations to use the generators of $\G(\calK)$ instead of the generators of $\G(\ch{\calK})$. Then the chirality group $X(\K)$ of
$\calK$ is defined to be the kernel of the cover from $\G(\ch{\calK})$ to $\Delta$. Equivalently, if
$\G(\calK)$ has presentation $\langle \s_1, \ldots, \s_{d-1} \mid \calR \rangle$, where $\calR$ is the set of
defining relators of $\G(\calK)$, then
\[ X(\calK) = \langle \ch{\calR} \rangle^{\G(\calK)}, \]
the normal closure of $\ch{\calR}$ in $\G(\calK)$, where $\ch{\calR} := \{\ch{w} \mid w \in \calR\}$.

   It was proved in \cite{chiregas} that any group $\Gamma$ together with a generating set
$\{\sigma_1, \dots, \sigma_{d-1}\}$ satisfying (\ref{sigmarelation}) and the intersection
conditions in (\ref{intersectprop}) is the rotation subgroup of a polytope $\calP$ that is either
orientably regular or chiral. Furthermore,
if $\langle \s_1, \ldots, \s_{d-2} \rangle$ is the automorphism group of a chiral polytope (in which
case $\calP$ has chiral facets), then $\calP$ is itself chiral.

	In order to check the relations (\ref{sigmarelation}), we will use \cite[Lemma 1]{higherranks}:

\begin{lemma}\label{sigmastau}
   Let $d \ge 4$, let $\Gamma$ be a group generated by elements $\sigma_1,
\dots, \sigma_{d-1}$, and let $\tau_{d-2, d-1} := \sigma_{d-2} \sigma_{d-1}$.
If the relations (\ref{sigmarelation})
hold for every $j<d-1$, then the set of relations
$(\sigma_k \cdots \sigma_{d-1})^2 = \varepsilon$,
$k = 1, \dots, d-2$ is equivalent to the set of relations
\begin{eqnarray}\label{equa1}
&&\tau_{d-2, d-1}^2 = \varepsilon\\ \label{equa2}
&&\tau_{d-2, d-1} \sigma_{d-3} \tau_{d-2, d-1} = \sigma_{d-3}^{-1}\\ \label{equa3}
&&\tau_{d-2, d-1} \sigma_{d-4} \tau_{d-2, d-1} = \sigma_{d-4} \sigma_{d-3}^2\\\label{equa4}
&&\tau_{d-2, d-1} \sigma_j = \sigma_j \tau_{d-2, d-1} \quad \mbox{for $j < d-4$.}
\end{eqnarray}
\end{lemma}

Note that (\ref{equa2}), (\ref{equa3}) and (\ref{equa4}) are equivalent to requiring that  $\tau_{d-2,d-1} \s_j$ equals $\overline{\s_j} \tau_{d-2, d-1}$ for $j \in \{1, \dots, d-3\}$, where the enantiomorphic word is taken with respect to the $(d-2)$-faces.

   The following criterion determines that some suitable groups  $\Gamma$ satisfy the intersection
condition (\ref{intersectprop}) and follows directly from \cite[Lemma 10]{chiregas}. For the purposes of this paper we only need the version where the facets of the chiral $d$-polytope corresponding to $\Gamma$ are themselves chiral $(d-1)$-poytopes.

\begin{lemma}\label{intersecteor}
   Let $n \ge 4$ and let $\Gamma = \langle \sigma_1, \dots, \sigma_{d-1} \rangle$ be a
group satisfying (\ref{sigmarelation}). If
$\langle \sigma_1, \dots, \sigma_{d-2} \rangle$ is the automorphism group of a chiral $(d-1)$-polytope with regular facets,
and if the intersection condition
\[\langle \sigma_1, \dots, \sigma_{d-2} \rangle \cap \langle \sigma_k, \dots,
\sigma_{d-1} \rangle = \langle \sigma_k, \dots, \sigma_{d-2} \rangle\]
holds for $k = 2, \dots, d-1$, then $\Gamma$ satisfies the intersection condition (\ref{intersectprop}).
\end{lemma}

   For any $d$-polytope $\K$ (not necessarily regular or chiral) and $i=1, \dots, d-1$
we define the (involutory) permutation $r_i$ on the flags of $\K$ by
$\Phi r_i = \Phi^i$. (That is, $r_i$ moves \emph{every} flag to its $i$-adjacent flag.)
The subgroup $Mon(\K) := \langle r_0, \dots, r_{d-1} \rangle$ of the symmetric
group on the set of flags of $\K$ is often
referred to as the {\em monodromy group} of $\K$ (see for example
\cite{selfinvariance}). Note that for any flag $\Phi$ of $\K$ and any
automorphism $\varphi \in \Gamma(\K)$, the definition of an automorphism of $\K$
implies that $(\Phi \varphi) r_i = (\Phi r_i) \varphi$. An inductive argument
can be used to show that, for any word $w$ on the generators $r_i$,
\begin{equation}\label{monaut}
(\Phi \varphi) w = (\Phi w) \varphi.
\end{equation}

It is well-known that if $\K$ is regular then $Mon(\K) \cong \Gamma(\K)$ (see \cite[Theorem 3.9]{MixMon}). On the other hand, if $\K$ is chiral then its automorphism group is related with its monodromy group in a slightly different way. Let $Mon(\K) = \langle r_0, \dots, r_{d-1} \rangle$ and let $s_i := r_{i-1} r_i$ for $i = 1, \dots, d-1$. The next proposition is essentially \cite[Proposition 7]{higherranks}.

\begin{proposition}\label{monoautid2}
   Let  $\K$ be a chiral $d$-polytope, let $\mathcal{O}$ be the orbit of the
base flag $\Phi$ of $\K$ under $\Gamma(\K)$,
and let $s_1, \dots, s_{d-1}$ be as above. Then there is a
group isomorphism between $\Gamma(\K)$ and the
permutation group on $\mathcal{O}$ induced by
$\langle s_1, \dots, s_{d-1} \rangle$, mapping $\sigma_i$ to the permutation induced by $s_i$.
\end{proposition}

   Some examples of orientably regular and chiral polytopes are given by toroidal quotients
of the tessellation $\{4, 4\}$ of the Euclidean
plane by squares. Let $T$ denote the group of
translations of the plane which fixes this tessellation.
We denote by $\{4, 4\}_{\bar{u}}$ the quotient of the regular tessellation
$\{4, 4\}$ by the subgroup $T_{\bar{u}}$ of $T$ generated by the
translational symmetries with respect to the linearly independent vectors $\bar{u}$ and
$\bar{u} R$, where $R$ stands for the rotation by $\pi/2$. The chiral polyhedra $\{4,4\}_{(2,1)}$ and $\{4,4\}_{(4,1)}$ are shown in Figure \ref{toroidalmaps}.

\begin{figure}
\begin{center}
\includegraphics[width=5.5cm, height=3cm]{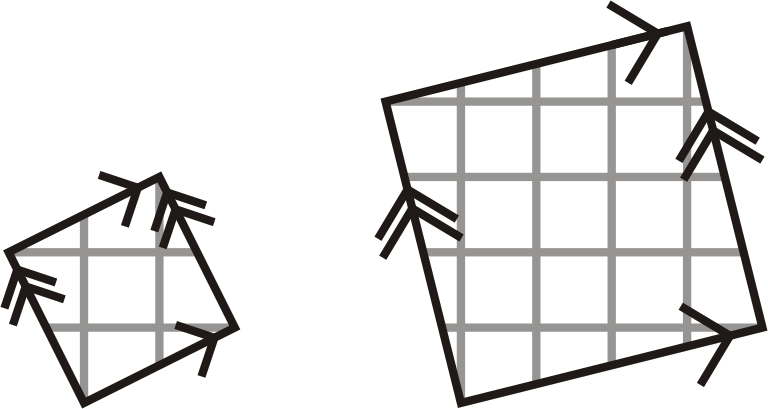}
\end{center}
\caption{Toroidal polyhedra \label{toroidalmaps}}
\end{figure}

It is well known that the toroidal map
$\{4, 4\}_{\bar{u}}$ is regular whenever $\bar{u} = (a, 0)$
or $\bar{u} = (a, a)$ with the vectors taken with
respect to the basis $\{e, e R\}$ where $e$ denotes a vector with the size and
direction of an edge of the tessellation.
Moreover, the toroidal map $\{4, 4\}_{\bar{u}}$ is
chiral whenever it is not regular (see \cite{CoM}, \cite[Section 1D]{ARP}). Note that
$\{4, 4\}_{(1, 0)}$ and $\{4,4\}_{(1,1)}$ fail to be
strongly flag-connected, and thus, they are not abstract polyhedra.



\section{GPR graphs}
\label{SectionGraphs}

   CPR (C-group permutation representation) graphs are introduced in \cite{pre}
as graphs encoding all the information of
the automorphism group of a regular polytope. This concept is extended in
\cite{pewe} to GPR (general permutation representation) graphs encoding all
information of the rotation subgroup of orientably regular or chiral polytopes.
In this section we review the definition and some basic results about GPR graphs of orientably regular or chiral polytopes.

   Let $\K$ be a finite orientably regular or chiral polytope with $\Gamma^+(\K) = \langle \sigma_1,
\dots, \sigma_{d-1} \rangle$ and let $\pi$ be an embedding of $\Gamma^+(\K)$ into
a symmetric group $S_t$. The {\em GPR graph of $\K$ determined by $\pi$} is the
directed multigraph (parallel arrows are allowed) with vertex set $\{1, \dots,
t\}$ and an arrow labeled $k \in \{1, \dots, d-1\}$ from $u$ to $v$ if and
only if $v = u(\sigma_k \pi)$. In most cases GPR graphs are Schreier's coset diagrams of $\Gamma(\K)$ as
defined in \cite[Section 3.7]{CoM}; however, we allow GPR graphs to be disconnected.

   We shall denote arrows labeled $k$ as {\em $k$-arrows}.
For simplicity we omit all loops, and therefore, fixed points by the image
under $\pi$ of $\sigma_k$ are determined by the vertices contained in no $k$-arrow.
Note that, for every $k \in \{1, \dots, d-1\}$, the $k$-arrows
form directed cycles.
When convenient we shall refer to the connected components induced by arrows
labeled $k$ and by arrows labeled $j, j+1, \dots, k$ as {\em $k$-components} and
{\em $(j, k)$-components} respectively.

   Whenever the embedding $\pi$ plays
no role or it is clear from the context we shall refer to the GPR graph of $\K$
determined by $\pi$ simply by a GPR graph of $\K$.

   The group $\Gamma^+(\K)$ acts naturally on the vertex set of the GPR graph of
$\K$ determined by $\pi$. This action is defined by $v \sigma_k := v (\sigma_k \pi)$
and, since $\pi$ is an embedding, the action is faithful. Consequently, any GPR graph
of a regular or chiral polytope $\K$ completely determines the rotation subgroup
of $\K$.

   Whenever $\K$ is orientably regular or chiral, with base flag $\Phi$, we can
consider the natural action of $\G^+(\calK)$ on the flags in the same orbit as $\Phi$.
In other words, $\pi$ may be understood as the natural embedding of $\G^+(\calK)$ into the symmetric
group on this orbit. Then we
say that the GPR graph of $\K$ determined by $\pi$ is the {\em Cayley GPR
graph} of $\K$ and denote it by $Cay(\K)$. Note that $Cay(\K)$ is the (colored)
Cayley graph of $\Gamma^+(\K)$ with generators $\sigma_1, \dots, \sigma_{d-1}$
in the sense of \cite{white}. The following proposition and corollary follow from the
regular action (transitive action with trivial point stabilizers)
of $\Gamma^+(\K)$ on one of the two induced orbits of flags of $\K$, or alternatively,
from the regular action of a group on any of its Cayley graphs.

\begin{proposition}\label{uniquecayley}
   Let $G$ be the Cayley GPR graph of an orientably regular or
chiral $d$-polytope $\K$, and let $u$ and $v$ be vertices of $G$. Then there exists a
unique element $\alpha \in \Gamma(\K)$ such that $u \alpha = v$.
\end{proposition}

\begin{corollary}\label{idcayley}
   Let $G$ be the Cayley GPR graph of an orientably regular or
chiral $d$-polytope $\K$. Then $\varepsilon$ is the only element in $\Gamma(\K)$
whose action on the vertices of $G$ has fixed points.
\end{corollary}

There is a standard way of representing the Cayley GPR graph of an orientably regular or chiral polytope $\K$.
Consider as vertices the flags in the orbit of the base flag under $\Gamma(\K)$, and add arrows labeled $k$ from each
flag $\Phi$ to flag $(\Phi^{k-1})^{k}$. That is, the arrows labeled $k$ correspond to the action of the monodromy element
$s_k$ as defined in Section \ref{SectionDefinitions}. By Proposition \ref{monoautid2}, the graph obtained this way
is isomorphic to the Cayley GPR graph of $\K$.

	Given an embedding $\pi: \G^+(\calK) \to S_t$, there is a natural embedding $\pi^n: \G^+(\calK) \to S_t \times \cdots
\times S_t$ into the direct product of $n$ copies of $S_t$, defined by $\s_k \pi^n = (\s_k \pi, \ldots, \s_k \pi)$.
The following proposition follows immediately.

\begin{proposition}\label{isoremark}
	Let $G_1, \ldots, G_m$ be isomorphic copies of a GPR graph of the polytope $\calK$.
Then the disjoint union of all the graphs $G_j$ is a GPR graph of $\calK$.
\end{proposition}

	If $G$ is a graph with vertex set $\{1, \ldots, t\}$ and arrows with labels in $\{1, \ldots, d-1\}$
such that no vertices are incident to two or more arrows with the same label, then there are naturally
induced permutations $\s_1, \ldots, \s_{d-1}$ and a group $\G = \langle \s_1, \ldots, \s_{d-1} \rangle$.
In order to determine when $G$ is the GPR graph of an orientably regular or chiral polytope $\calK$,
we need to check whether $\G$ satisfies the relations (\ref{sigmarelation}) and the intersection condition
(\ref{intersectprop}). Furthermore, we would like to determine whether $\calK$ is orientably regular or chiral.
In general, it is not obvious from looking at $G$ whether these properties hold, but the
following variation of \cite[Theorem 3.7]{higherranks} will be useful for our construction.

\begin{theorem}\label{cayleyintersection}
   Let $G$ be a graph with arrows labeled $1, \dots, d$ and let $G_1, \dots, G_s$ be
the $(1, d-1)$-components of $G$. Assume also that
\begin{itemize}
   \item[(a)] $G_1, \dots, G_s$
         are isomorphic to the Cayley GPR graph of a fixed chiral $d$-polytope $\K$,
   \item[(b)] for $k = 1, \dots, d-1$, the action of
         $(\sigma_k \cdots \sigma_{d})^2$ on
         the vertex set of $G$ is trivial, where $\sigma_i$ is the permutation
         determined by all arrows of label $i$,
   \item[(c)] $\langle \sigma_1, \dots, \sigma_{d-1} \rangle \cap \langle \sigma_{d}
         \rangle = \{\varepsilon\}$.
   \item[(d)] for every $k = 2, \dots, d-1$ there exists a $(1, d-1)$-component $G_{i_k}$
         and a $(k, d)$-component $D_k$ such that $G_{i_k} \cap D_k$ is a nonempty
         $(k, d-1)$-component.
\end{itemize}
   Then $G$ is a GPR graph of a chiral $(d+1)$-polytope $\calP$.
\end{theorem}

\begin{proof}
The proof is analogous to \cite[Theorem 3.7]{higherranks}, but due to its importance, we present a full proof here.
If we can show that the group $\langle \s_1, \ldots, \s_{d} \rangle$ induced by $G$ satisfies the relations (\ref{sigmarelation})
and the intersection condition (\ref{intersectprop}), then it will follow that $G$ is a GPR graph of
an orientably regular or chiral polytope $\calP$. Furthermore, since the $(1, d-1)$-components of $G$ are all isomorphic to
$Cay(\calK)$, and $\calK$ is chiral, it will follow that the facets of $\calP$ are
chiral, and so $\calP$ itself is chiral.

The relations (\ref{sigmarelation}) follow from $(a)$, $(b)$ and the faithful
action of $\langle \sigma_1, \dots, \sigma_{d-1} \rangle$ on each $G_i$. Since the components $G_i$ are all
isomorphic, Proposition~\ref{isoremark} implies that $\langle \s_1, \ldots, \s_{d-1} \rangle$ is isomorphic
to the group determined by $G_{i_1}$, which is $\G(\calK)$. Then by Lemma~\ref{intersecteor}, we only need to show
that for $2 \le k \le d$,
\[\langle \sigma_1, \dots,
\sigma_{d-1} \rangle \cap \langle \sigma_k, \dots, \sigma_{d} \rangle =
\langle \sigma_k, \dots, \sigma_{d-1} \rangle.\]
The case $k = d$ follows from $(c)$. Let $2 \le k \le d-1$,
$\alpha \in \langle \sigma_1, \dots, \sigma_{d-1} \rangle \cap \langle \sigma_k,
\dots, \sigma_{d} \rangle$, and $v$ a vertex of $G_{i_k} \cap D_k$. Since
$\alpha \in \langle \sigma_1, \dots, \sigma_{d-1} \rangle$, $v \alpha$
belongs to $G_{i_k}$. A similar argument shows that $v \alpha$ belongs to
$D_k$. By $(d)$ it follows that $v \alpha$
and $v$ belongs to the same $(k, d-1)$-component, namely $G_{i_k} \cap D_k$,
and therefore there exists $\beta \in \langle \sigma_k, \dots,
\sigma_{d-1} \rangle$ such that $v \alpha \beta = v$. Since $\alpha \beta \in
\langle \sigma_1, \dots, \sigma_{d-1} \rangle$ it follows from Corollary
\ref{idcayley} that
$\alpha = \beta^{-1} \in \langle \sigma_k, \dots,
\sigma_{d-1} \rangle$.
\end{proof}

   As we shall see, sometimes it is convenient to consider an alternative generating
set for the rotation subgroup
$\Gamma^+(\K) = \langle \sigma_1, \dots, \sigma_{d-1} \rangle$
of a chiral $d$-polytope $\K$. In Section
\ref{SectionResult} we shall make use of
\[\{\sigma_1, \dots, \sigma_{d-2}, \tau_{d-2, d-1}\}\]
as a generating set for $\Gamma^+(\K)$. In so doing, we consider GPR graphs with
arrows labeled $1, \dots, d-2$ corresponding to $\sigma_1, \dots, \sigma_{d-2}$, and
a matching corresponding to
$\tau_{d-2, d-1}$ consisting of (non-oriented) edges between $v$ and $v \tau_{d-2, d-1}$
for each vertex $v$. This will make it immediate to verify relation (\ref{sigmarelation})
for $i= d-2$, $j=d-1$.

To conclude this section we state Jordan's theorem on primitive permutation groups (see \cite[Theorem 13.9]{Wielandt}).

\begin{theorem}\label{t:jordan}
Let $\Lambda$ be a primitive permutation group on the set $\{1, \dots, n\}$. If $\Lambda$ contains a $p$-cycle
for some prime $p < n-2$ then $\Lambda$ is either $S_n$ or $A_n$.
\end{theorem} 

\section{Chiral extensions of chiral polytopes}
\label{SectionResult}
Let $\K$ be a chiral $d$-polytope with regular facets. In this section, we will provide a construction of a
(GPR graph of a) chiral $(d+1)$-polytope $\Po$ with facets isomorphic to $\K$.

To start the construction of a GPR graph $G_\Po$ of $\Po$, we take two disjoint copies $G_\K$ and $G_\K'$ of the Cayley GPR graph of $\K$ with standard generators $\sigma_1, \dots, \sigma_{d-1}$. Let us denote the vertices of $G_\K$ by $\{v_1, \dots, v_t\}$ and the vertices of $G_\K'$ by $\{v_1', \dots, v_t'\}$ in such a way that there is an arrow labelled $k$ from $v_i$ to $v_j$ if and only if there is an arrow labelled $k$ from $v_i'$ to $v_j'$. In other words, the labels of $G_\K'$ are obtained from those of $G_\K$ just by attaching a prime. The idea is to add a perfect matching between $G_\K$ and $G_\K'$, where the new edges correspond to the new generator $\tau_{d-1,d}$.

Let $l$ be the order of $\sigma_{d-1}$ and let $v_1$ and $v_1'$ be the vertices of $G_\K$ and $G_\K'$ (respectively) corresponding to the base flag of $\K$. We add an edge of the matching between $v_1$ and $v_1'$, and for $k \in \{1, \dots, l-1\}$ we also match $v_1 \sigma_{d-1}^{k}$ to $v_1' \sigma_{d-1}^{l-k}$ (see the left part of Figure \ref{GPRconstruction} where an example for $l =5$ is given; the gray dotted lines represent $\tau_{d-1,d}$ whereas the arrows represent $\sigma_{d-1}$).

For $k = 1, \dots, d-2$, let us define $E_k$ to be the $(k, d-1)$-component containing $v_1$. In other words $E_k = v_1\langle\sigma_k, \dots, \sigma_{d-1}\rangle$. Note that if a $(1, d-2)$-component $A$ of $G_\K$ intersects $E_k$, then it also intersects $E_j$ for all $j < k$. Let us define
\[\mathcal{C}_k = \{(1,d-2)-\mbox{components $A$} \mid \mbox{$A$ intersects $E_k$ but not $E_{k+1}$}\}.\]
Now, for each $A \in \mathcal{C}_k$, we pick a vertex $v_i \in A \cap E_k$ and match it to $v_i'$. The right part of
Figure~\ref{GPRconstruction} illustrates our strategy; the dashed ellipses represent the $(1, d-2)$-components of $G_\K$.

\begin{figure}
\begin{center}
\includegraphics[width=12cm, height=4.5cm]{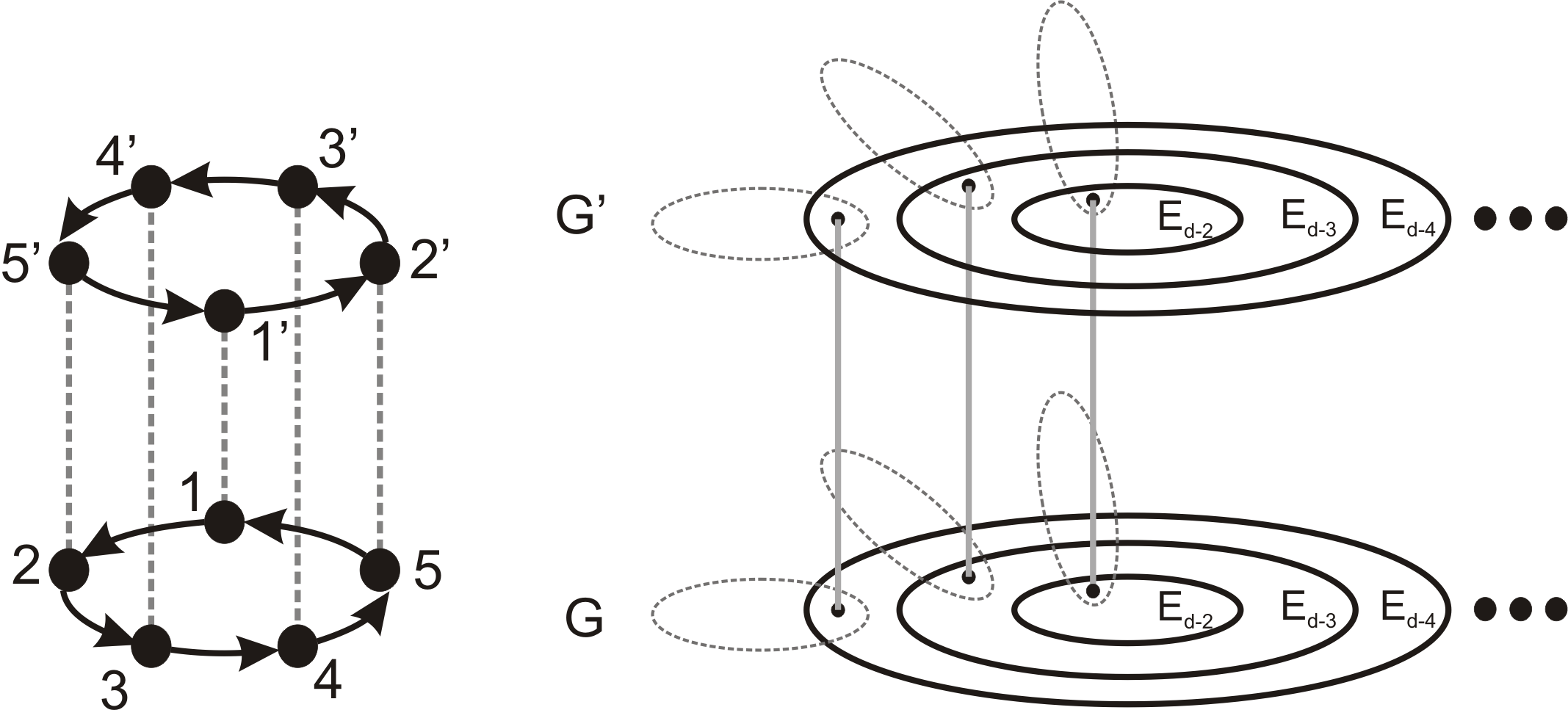}
\caption{Construction of a GPR graph of $\Po$}
\label{GPRconstruction}
\end{center}
\end{figure}

So far, we have defined the matching only on one vertex of each $(1,d-2)$-component of $G_\K$ and $G_\K'$. Furthermore,
each $(1,d-2)$-component $A$ has a vertex that has been matched, since $A$ either lies in one of the sets $\mathcal{C}_k$
or contains a single vertex of the form $v_1 \s_{d-1}^k$. Therefore, we have an induced
matching between the $(1,d-2)$-components of $G_\K$ and $G_\K'$. It remains to extend the matching to the remaining vertices. Note that once we have matched $v_i$, then the other vertices in the same $(1,d-2)$-component are of the form
$v_i \alpha$ for $\alpha \in \langle \s_1, \ldots, \s_{d-2} \rangle$, and so we only need to decide how we will match
such vertices. Now, if $\Q$ is the facet of $\K$ then $\alpha \in \G^+(\Q)$. Furthermore, $\Q$ is regular, so
$\G(\Q) = \langle \rho_0, \ldots, \rho_{d-2} \rangle$. Therefore, since $\alpha \in \G^+(\Q)$, it follows that
$\rho_{d-2} \alpha \rho_{d-2} \in \G^+(\Q)$. In fact, $\rho_{d-2} \alpha \rho_{d-2} = \ch{\alpha}$, where
$\ch{\alpha}$ is obtained from $\alpha$ by replacing each factor $\sigma_{d-2}$ by a factor $\sigma_{d-2}^{-1}$ and any factor
$\sigma_{d-3}$ by $\sigma_{d-3}\sigma_{d-2}^2$ in any word representing $\alpha$ in terms of $\sigma_1, \dots, \sigma_{d-2}$.
Now, we declare that whenever $v_i$ is matched to $v_j'$, then we must match $v_i \alpha$ to $v_j' \ch{\alpha}$.

We define the graph $G_\Po$ to be the disjoint union of $G_\K$ and $G_\K'$ together with the matching described above. The arrows of label $i$ induce a permutation $\eta_i$ on the disjoint union of the vertex sets of $G_\K$ and $G_\K'$, whereas the matching represents a permutation $\tau$ on the same set.

The next theorem shows that the graph just constructed describes a chiral $(d+1)$-polytope.

\begin{theorem}\label{MainProposition}
The graph $G_\Po$ defined above is a GPR graph of a chiral $(d+1)$-polytope with facets isomorphic to $\K$.
\end{theorem}

\begin{proof}
Let $\tau$ be the permutation induced by the matching between $G_\K$ and $G_\K'$ and, for $i \in \{1, \dots, d-1\}$, let $\eta_i$ be the permutation corresponding to arrows labelled $i$. By Proposition \ref{isoremark} we know that $\langle \eta_1, \dots, \eta_{d-1} \rangle \cong \langle \sigma_1, \dots, \sigma_{d-1} \rangle$.

To prove the claim, we will use Theorem~\ref{cayleyintersection}. Part (a) is true by construction. For part (b), we will
use Lemma \ref{sigmastau} and show that the permutations $\eta_i$ satisfy relations (\ref{equa1}), (\ref{equa2}), (\ref{equa3}) and
(\ref{equa4}). Relation (\ref{equa1}) is satisfied by construction of $\tau$, since the edges matching $G_\K$ with $G_\K'$ are
undirected. On the other hand, for $i \in \{1, \dots, d-2\}$, the action of $\tau \eta_i \tau$ on every vertex of $G_\Po$ is
precisely the one required in Lemma \ref{sigmastau}. (Observe that the subindices are shifted by one.)

For part (c), we note that
\begin{eqnarray*}
v_1 \eta_d^2 &=& v_1 (\eta_{d-1}^{-1} \tau)^2\\ &=& v_1 \eta_{d-1}^{-1} \tau \eta_{d-1}^{-1} \tau\\
             &=& v_1 \tau \eta_{d-1} \eta_{d-1}^{-1} \tau = v_1.
\end{eqnarray*}
This implies that every power of $\eta_d$ preserving $G_\K$ must fix $v_1$. Corollary \ref{idcayley} now proves that $\langle \eta_1, \dots, \eta_{d-1} \rangle \cap \langle \eta_{d} \rangle = \{\eps\}$.

To conclude the proof, it suffices to prove that, for every $k \in \{2, \dots, d-1\}$, the intersection of $G_\K$ with the $(k,d)$-component containing $v_1$ is $E_k$ (the $(k,d-1)$-component of $v_1$). In order to do this we show that, for every $k$, the vertices in the $(k,d-1)$-component of $v_1$ are matched to the vertices in the $(k,d-1)$-component of $v_1'$. In other words, we want to show that if $\alpha \in \langle \sigma_k, \dots, \sigma_{d-1} \rangle$, then there is an $\alpha' \in \langle \sigma_k, \dots, \sigma_{d-1} \rangle$ such that $v_1 \alpha$ is matched to $v_1' \alpha'$.

Let $k \in \{2, \dots, d-1\}$, let $\alpha \in \langle \sigma_k, \dots, \sigma_{d-1} \rangle$ and let $A$ be the
$(1,d-2)$-component that contains $v_1 \alpha$. Then $A$ intersects $E_k$, and in the construction of the matching we picked some
$v_i \in A \cap E_k$ and matched it to $v_i'$. (Indeed, $v_i$ may have been chosen in $A \cap E_j$ for some $j>k$, with no
consequences in the rest of the proof.) Now, since $v_i$ is in the $(1, d-2)$-component of $v_1 \alpha$ and the $(k,d-1)$-component
of $v_1$ (and therefore of $v_1 \alpha$), it follows that $v_i = v_1 \alpha \beta$ for some $\beta \in
\langle \s_1, \ldots, \s_{d-2} \rangle \cap \langle \s_k, \ldots, \s_{d-1} \rangle$ (and similarly,
$v_i' = v_1' \alpha \beta$). Since $\K$ is a polytope,
it follows that $\beta \in \langle \s_k, \ldots, \s_{d-2} \rangle$. In the last step of the construction we matched
$v_1\alpha = v_i \beta^{-1}$ with $v_i' \overline{\beta^{-1}} = v_1' \alpha \beta \overline{\beta^{-1}}$. But $\overline{\beta^{-1}}\in \langle \sigma_k, \dots, \sigma_{d-1} \rangle$ implying that $\alpha' = \alpha \beta \overline{\beta^{-1}} \in \langle \sigma_k, \dots, \sigma_{d-1} \rangle$, so our matching has the desired property.
\end{proof}

Theorem \ref{MainProposition} does not require the polytope $\K$ to be finite. However, in \cite{chiregas3} another construction was given of an infinite chiral $(d+1)$-polytope whose facets are any given chiral $d$-polytope with regular facets. The main contribution of Theorem \ref{MainProposition} is to add a finiteness property in the extension when $\K$ is finite, thus implying Theorem \ref{MainTheorem}. 

\section{Extending the toroidal maps $\{4, 4\}_{(b, c)}$}
\label{SectionExamples}
Here we will illustrate the construction of Theorem~\ref{MainProposition}
by extending the toroidal maps $\{4, 4\}_{(b, c)}$ to finite chiral
$4$-polytopes and examining the structure of the resulting automorphism groups.

Let $\calK := \{4, 4\}_{(b, c)}$,
and let $n = b^2 + c^2$ be prime (from which it follows that $n \equiv 1$ (mod 4)). We shall call {\em translations} the automorphisms generated by those mapping the base flag $\Phi$ to $\Phi^{1210}$ and $\Phi^{1012}$. These automorphisms extend naturally to the usual translations of the Euclidean tessellation $\{4,4\}$.

Note that since $n$ is prime,
and in particular $b$ and $c$ are coprime, it follows that we can reach every square using only horizontal translations.
Writing $X$ for translation to the right by one square and $Y$ for translation up by one square, there is
some $t$ such that $2 \leq t \leq n-2$ and $X^t = Y^{\pm 1}$. Furthermore, by picking $\bar{\K}$ instead of $\K$, the directions of the translations $X^t$ and $X^{n-t}$ are interchanged (see Figure \ref{flaglabels}), and we may choose $t$ such that it is either divisible by $4$ or congruent to $3$ mod $4$ (we will see later why this is useful).
In any case, if $X^t = Y^{\pm 1}$, then it also follows that $Y^t = X^{\mp 1}$ and so $X^{t^2} = Y^{\pm t} = X^{-1}$.
Since the order of $X$ is $n$, we see that $t^2 \equiv -1$ (mod $n$).

Let $G = Cay(\calK)$ with respect to the flag orbit containing $\Phi$.
Since $\calK$ has $n$ square faces, it has $8n$ flags, and $G$ has $4n$ vertices, corresponding to the {\em even
flags}, which are those that lie in the orbit $\Phi \Gamma(\K)$. We will number the flags of $G$ as follows. We label $\Phi$ as $1$, $\Phi^{10}$ as 2, $\Phi^{1010}$ as 3 and $\Phi^{01}$ as 4. Now, whenever we move right one square, we simply add $4$ to
all of the flag labels, setting by ``right'' the direction that takes $\Phi$ to $\Phi^{1210}$ (see Figure~\ref{flaglabels}).

\begin{figure}
\begin{center}
\includegraphics[width=11cm, height=5cm]{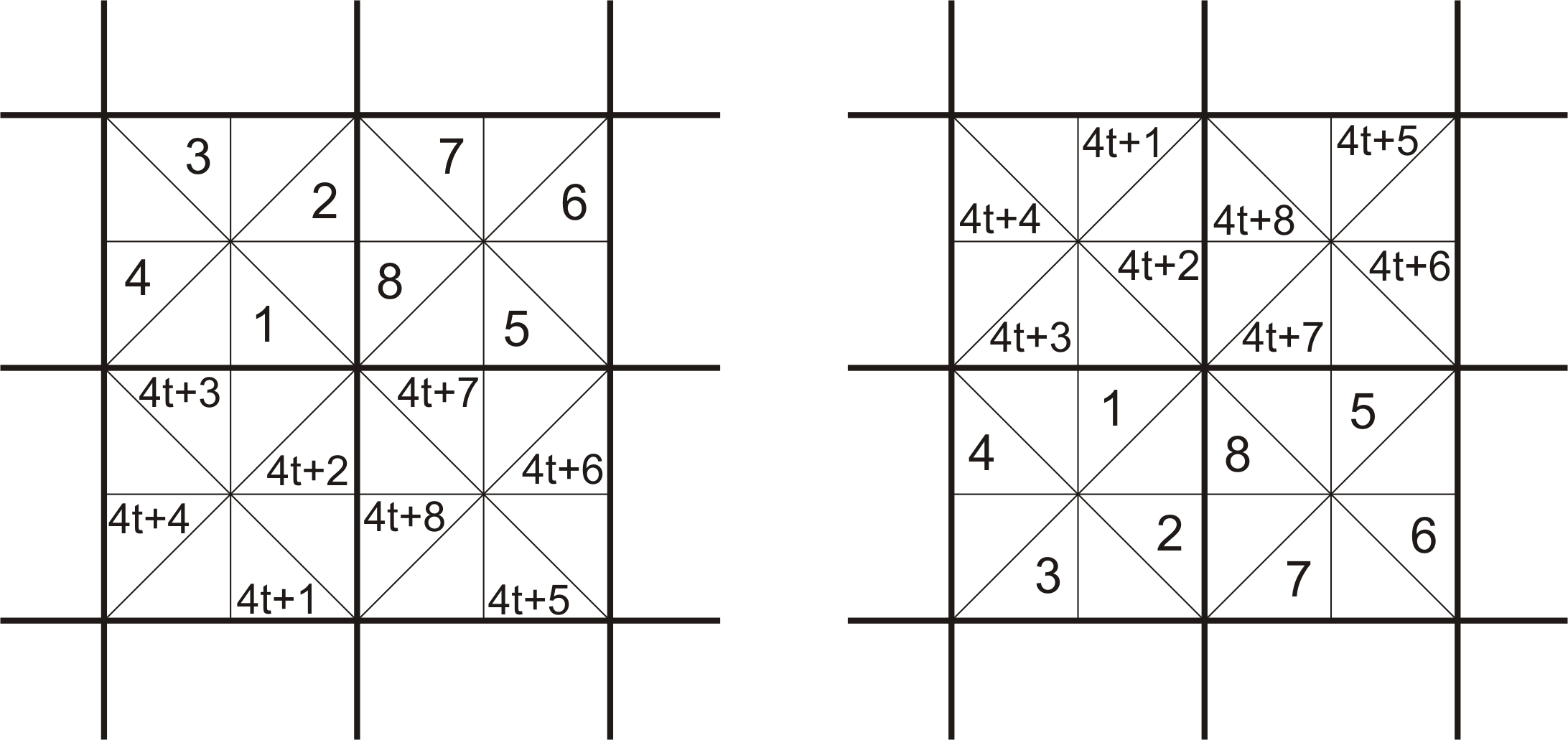}
\end{center}
\caption{Local numbering of the flags of $\{4, 4\}_{(b, c)}$ \label{flaglabels}}
\end{figure}

With our flags numbered in this way, Proposition \ref{monoautid2} allows us to assume that $\s_1$ and $\s_2$ act on the vertices of $G$ as follows:
\[ \s_1 = (1, 2, 3, 4)(5, 6, 7, 8) \cdots (4n-3, 4n-2, 4n-1, 4n) \]
\[ \s_2 = (1, \, 4t+2, \, 4t+7, \, 8) \cdots (4i+1, \, 4i+4t+2, \, 4i+4t+7, \, 4i+8) \cdots (4n-3, \, 4t-2, \, 4t+3, \, 4), \]
where in $\s_2$, we reduce each entry modulo $4n$. So we see that $\s_1$ permutes the flags in each square
in a cycle, and $\s_2$ permutes the flags around a given vertex in a cycle.
Defining $x := \s_2^{-1} \s_1$, we see that
\begin{align*}
x = (&1, \, 5, \, 9, \, \ldots, \, 4n-3) \\
 (&2, \, 4n - 4t + 2, \, 4n - 8t + 2, \, \ldots, \, 4t + 2) \\
 (&3, \, 4n-1, \, 4n-5, \, \ldots, \, 7) \\
 (&4, \, 4+4t, \, 4+8t, \, \ldots, \, 4n -4t + 4).
\end{align*}
Note that $x$ is not merely a translation to the right by one square; it acts that way on flags of the form
$4i+1$, but on other sets of flags it acts as a translation in different directions.
This is because we are working with the action of the even subgroup of the monodromy group (the group generated by the $s_i$'s) instead of the rotation
subgroup of the automorphism group.


Now we proceed to the extension. In addition to the GPR graph $G$, we take another copy $G'$ of $Cay(\K)$ that is labeled in the
same way as $G$ except that we denote the vertices by $1', \ldots, (4n)'$. Furthermore, we extend $\s_1$ and
$\s_2$ so that they act on these new vertices in exactly the same way as the original vertices. That is, if
$\alpha \in \langle \s_1, \s_2 \rangle$, then
\[ (i') \alpha = (i \alpha)'. \]
We define $\tau := \tau_{2,3}$ as follows:
\begin{align*}
\tau = &\left(5, (4t+1)'\right)\left(6, (4t+4)'\right)\left(7, (4t+3)'\right)\left(8, (4t+2)'\right) \cdot\\
 &\left(4t+1, 5'\right)\left(4t+2, 8'\right)\left(4t+3, 7'\right)\left(4t+4, 6'\right) \cdot\\
& \prod_{0 \leq i \leq n-1, i \neq 1, t} \left(4i+1, (4i+1)'\right)\left(4i+2, (4i+4)'\right)\left(4i+3, (4i+3)'\right)
	\left(4i+4, (4i+2)'\right).
\end{align*}
Then $\tau$ induces a matching between the vertices of $G$ and those of $G'$, and by including those undirected edges we get
a graph with vertex set $V = \{i, i' \mid 1 \leq i \leq 4n\}$. We shall abuse notation and denote respectively by $\sigma_1$ and $\sigma_2$ the permutations of $V$ induced by the action of $\sigma_1$ and $\sigma_2$ on the vertices of $G$ and $G'$.

\begin{theorem}\label{t:4poly}
Let $\calK$, $\s_1$, $\s_2$ and $\tau$ be defined as in the preceding discussion, with $n = b^2 + c^2$ prime.
Let $\G = \langle \s_1, \s_2, \tau \rangle$. Then $\G$ is the automorphism group of a finite chiral
polytope with Schl\"afli symbol $\{4, 4, 2(n-2)\}$ and facets isomorphic to $\calK$.
\end{theorem}

\begin{proof}
The permutation $\tau$ induces a perfect matching between $G$ and $G'$. In particular, it matches $1$ to $1'$ and $1 \s_2^k$
to $1' \s_2^{-k}$, and whenever it matches $i$ to $j'$, it also matches $i \s_1^k$ to $j' \s_1^{-k}$. Therefore, the
matching induced by $\tau$ follows the ge\-ne\-ral construction in Section~\ref{SectionResult}, and Theorem~\ref{MainProposition}
shows that $\G$ is the automorphism group of a chiral polytope $\calP$ with facets isomorphic to $\calK$.

Next we find the order of $\s_3 = \s_2^{-1} \tau$. Since $\s_3$ sends vertices in $G$ to vertices in $G'$, the permutation
$\s_3$ must have even order. Consider the action of $\s_3^{-2}$ on vertices of the form $4i+1$. By construction,
$\s_3^{-2}$ fixes vertices $1$ and $4t+1$. A simple calculation then shows that $\s_3^{-2}$ sends $5$ to $8t-3$ and $5-4t$ to $4t-3$,
and that any other vertex $4i+1$ is sent to $4i+4t-3$. In other words, $\s_3^{-2}$ adds $4t-4$ to most vertices
of type $4i+1$, and it adds $8t-8$ to the vertices $5$ and $5-4t$, which is just enough to skip over the two fixed points.
Since there are $n$ vertices of the form $4i+1$, the permutation $\s_3^{-2}$ acts as an $(n-2)$-cycle on the vertices
of the form $4i+1$ other than $1$ and $4t+1$. Similar arguments can be used to show that $\s_3^{-2}$ acts
as an $(n-2)$-cycle on every other type of vertex (fixing two points), and so the order of $\s_3$ is $2(n-2)$.
Then, since the facets of $\calP$ are isomorphic to $\calK$, it follows that $\calP$ has Schl\"afli symbol
$\{4, 4, 2(n-2)\}$.
\end{proof}

We now determine the structure of $\G$.

For $k \in \{1, 2, 3, 4\}$, let us say that a vertex is in class $k$ (or in a {\em corner} of type $k$) if it has the form $4i + k$ for some $i$, and
in class $k'$ if it has the form $(4i + k)'$. Note that each class $k$ or $k'$ consists of all
of the vertices in $G$ or $G'$ (respectively) in the same orbit under the translation group of $\K$. The following lemma uses these classes
to factor $\G$ by a dihedral group $D_4$ with $8$ elements.

\begin{lemma}\label{l:corners}
Let $\Po$ be the $4$-polytope in Theorem \ref{t:4poly}. Then $\Gamma(\Po) \cong \Lambda \rtimes D_4$, where $\Lambda$ denotes the subgroup of $\G(\Po)$ stabilizing all classes of vertices of the GPR graph constructed above for $\Po$.
\end{lemma}

\begin{proof}
The group $\G$ has a well-defined action on the vertex classes; namely, $\s_1$ and $\s_2$ send class $i$ to class $i+1$ and
class $i'$ to class $(i+1)'$, while $\tau$ sends class $1$ to class $1'$, class $2$ to class $4'$, class $3$ to
class $3'$, and class $4$ to class $2'$.
So these 8 classes provide a block system for $\G$. 
Clearly, $\Lambda$ is a normal subgroup of $\G$ and $\G / \Lambda$ is the induced permutation group
on the blocks. Since $\s_1$ and $\s_2$ cyclically permute the blocks and $\tau \s_1 \tau$ acts as $\s_1^{-1}$, it follows
that $\langle \sigma_1, \tau \rangle$ is dihedral of order $8$ and it intersects $\Lambda$ trivially. Hence $G \cong \Lambda \rtimes \langle \sigma_1, \tau \rangle \cong \Lambda \rtimes D_4$.
\end{proof}

Consider the permutation
\begin{align*}
R := \prod_{i=0}^{n-1} & (4i+1, 4t-4i+7)(4i+2, 4t-4i+8) \\
& (4i+3, 4t-4i+5)(4i+4, 4t-4i+6) \\
& \left((4i+1)', (4t-4i+7)'\right)\left((4i+2)', (4t-4i+8)'\right) \\
& \left((4i+3)', (4t-4i+5)'\right)\left((4i+4)', (4t-4i+6)'\right).
\end{align*}
Geometrically, we can view $R$ as a half turn through the vertex at the center of Figure~\ref{flaglabels}, both in $G$ and in $G'$.
Furthermore, consider the permutation
\[ p := (1, 1') \cdots (4n, (4n)'). \]
Then the orbits of the vertices under $\langle R, p \rangle$ have size $4$, and it is easily verified that
these form a block system for $\G$. In other words, $\G$ permutes these blocks and the elements within each block.
The same is true for the subgroup $\Lambda$ of $\G$.

Now we observe that there are two kinds of blocks: those with
one element from each of the classes $1$, $1'$, $3$, and $3'$, and those with one element from each
of the classes $2$, $2'$, $4$, and $4'$. Since $\Lambda$ stabilizes each of these classes, it follows that it is merely a permutation group on the set of these blocks. Furthermore,
since $n$ of the blocks contain a vertex in class $1$ and $n$ contain a vertex in class $2$, the group
$\Lambda$ is isomorphic to a subgroup of $S_n \times S_n$.

Let $\lambda \in \Lambda$. Then we can write
\begin{equation}\label{eq:lambdaalpha}
\lambda = \alpha_1 \tau \alpha_2 \tau \cdots \alpha_r \tau,
\end{equation}
where each $\alpha_i \in \langle \s_1, \s_2 \rangle$. Furthermore, we may assume that each $\alpha_i \in \Lambda$,
since otherwise we can use the fact that $\tau = \s_1 \tau \s_1$ to replace the subword $\alpha_i \tau \alpha_{i+1}$
with $\alpha_i \s_1^s \tau \s_1^s \alpha_{i+1}$, where $s$ is chosen so that $\alpha_i \s_1^s \in \Lambda$. Now, $x = \sigma_2^{-1} \sigma_1\in \langle \s_1, \s_2 \rangle$ stabilizes each block and acts transitively
on the vertices in each block. Then, since $\langle \s_1, \s_2 \rangle$ acts regularly on the vertices of $\calK$, we have that $\langle x \rangle$ acts regularly on the set of vertices in any given block, and it follows that each $\alpha_i$ must be a power of $x$. Therefore,
$\lambda$ can be written as a product of powers of $x$ and $\tau x \tau$, and so
$\Lambda = \langle x, \tau x \tau \rangle$. (Note here that, since $\lambda \in \Lambda$, there must be an even number of factors $\tau$ in (\ref{eq:lambdaalpha}).)

We note that $x$ sends vertex
$4i+1$ to $4i+5$. Conjugating $x$ by $\tau$ causes us to interchange
$5$ with $4t+1$ in the cyclic structure, while preserving every other vertex of the form $4i+1$. In other words, while $x$ permutes
the vertices in class $1$ in the cycle $(1, 5, 9, \ldots, 4n-3)$, the permutation $\tau x \tau$ permutes the vertices
in the cycle
\[(1, 4t+1, 9, \ldots, 4t-3, 5, 4t+5, 4t+9, \ldots, 4n-3).\]
A similar analysis shows that each of $x$ and $\tau x \tau$ permutes the vertices in class $2$ in a single $n$-cycle.
Since $x$ and $\tau x \tau$ both
act on classes $1$ and $2$ as odd cycles, we see that $\Lambda$ is isomorphic to a subgroup of $A_n\times A_n$. We have then proved the following.

\begin{lemma}\label{l:directproduct}
Let $\Po$  be the $4$-polytope in Theorem \ref{t:4poly} and let $\Lambda$ be the subgroup of $\G(\Po)$ stabilizing all classes of vertices of the GPR graph constructed above for $\Po$.
Then $\Lambda \le A_n \times A_n$.
\end{lemma}

Each element $\lambda \in \Lambda$ can be represented as $\lambda = (\lambda_1, \lambda_2)$ where $\lambda_i$ represents the permutation induced on flags type $i$ by $\lambda$. We denote by $\Lambda_1 = \{(\lambda_1, \eps) \in \Lambda\}$ and $\Lambda_2 = \{(\eps, \lambda_2) \in \Lambda\}$. Then
\begin{equation}\label{eq:lambdaproduct}
\Lambda_1 \times \Lambda_2 \le \Lambda.
\end{equation}

We are now ready to determine $\G(\Po)$.

\begin{theorem}
Let $\Po$  be the $4$-polytope in Theorem \ref{t:4poly}
and assume that $b^2 + c^2 = n \geq 13$, with $n$ prime. 
Then $\G(\Po) \cong (A_n \times A_n) \rtimes D_4$.
\end{theorem}

\begin{proof}
The group $\Lambda_1$ is the subgroup of $\Lambda$ that fixes all vertices of the form $4i+2$ while $\Lambda_2$ is the subgroup of
$\Lambda$ that fixes all vertices of the form $4i+1$. Then $\Lambda_2 = \s_2^{-1} \Lambda_1 \s_2$. In particular, $\Lambda_1 \cong \Lambda_2$,
so it suffices to determine the structure of $\Lambda_2$. Our goal is to show that $\Lambda_2 \cong A_n$.
The group $\Lambda_2$ acts primitively
on the vertices in class $2$ (since the number of vertices, $n$, is prime).
It will then suffice to show that $\Lambda$ contains an element that acts trivially on class $1$ and as a $3$-cycle on
class $2$. It will follow from Theorem \ref{t:jordan} that $\Lambda_2 \cong A_n$. Finally, Lemma \ref{l:directproduct} and (\ref{eq:lambdaproduct}) will then show that $\Lambda \cong A_n \times A_n$, and Lemma \ref{l:corners} concludes the proof.

Let us define $z := x \tau x \tau$.
On most vertices of the form $4i+1$, the element $z$ acts like $x^2$, carrying $4i+1$ to $4i+9$.
The only exceptions are that it sends vertex $1$ to $4t+5$, vertex $4t-7$ to $5$, vertex $4t-3$ to $9$, and vertex $4n-3$ to $4t+1$.
Then if $t$ is even, we get that, when restricted to corners type 1,
\begin{align*}
z =  \,& (5, \, 13, \, \ldots, \, 4t-3, \, 9, \, 17, \, \ldots, \, 4t-7) \\
& (1, \, 4t+5, \, 4t+13, \, \ldots, \, 4n-7) \\
& (4t+1, \, 4t+9, \, \ldots, \, 4n-3),
\end{align*}
and if $t$ is odd, then
\begin{align*}
z = \,& (1, \, 4t+5, \, 4t+13, \, \ldots, \, 4n-3, \, 4t+1, \, 4t+9, \, \ldots, \, 4n-7) \\
& (5, \, 13, \, \ldots, \, 4t-7) \\
& (9, \, 17, \, \ldots, \, 4t-3).
\end{align*}
In the first case, the cycle sizes are $t-1$, $(n-t+1)/2$ and $(n-t+1)/2$, and in the second case they are $n-t+1$,
$(t-1)/2$ and $(t-1)/2$. Since we picked $t$ congruent to either $3$ or $0$ mod $4$, we see that in any case we get
three odd cycles, two of which have the same length. Let $m_1$ be the length of the first cycle, and let $m_2$ be the
length of the other two cycles. Since $n = m_1 + 2m_2$, and $n$ is prime, it must be the case that $m_1$ and $m_2$
are coprime.

We now determine the action of $z$ on class $2$. The action of $\tau x \tau$ on class $2$ is like the action of
$x$, but the cycle is reversed, and $6$ is interchanged with $4t+2$ as elements in the cyclic structure. Then $z$ fixes most vertices in
class $2$, with the exception of vertices $6$, $4t+2$, $4t+6$ and $8t+2$. Here we need to be careful to make sure
that this last vertex is actually different from the others. This follows from strong flag-connectivity, except for vertices $6$ and $8t+2$; however, if these vertices were the same then $\K \cong \{4,4\}_{(1,2)}$ and $n=5$, contradicting our hypothesis.
Then what we find is that $z$ interchanges vertex $6$ with $4t+2$ and vertex $4t+6$ with $8t+2$.
Therefore, the element $z$ is the product of a cycle of length $m_1$, two cycles of length $m_2$, and two
cycles of length $2$. In particular, when restricted to corners type $1$ and $2$ we have that $z^{m_1 m_2} = (6, 4t+2)(4t+6, 8t+2)$ since
$m_1$ and $m_2$ are both odd.

Let us define the permutation $w$ by
\[ w := \begin{cases}
\s_1^{-1} z^{m_2} \s_1 & \textrm{if $t$ is even,} \\
\s_1^{-1} z^{m_1} \s_1 & \textrm{if $t$ is odd.}
\end{cases} \]
On class $1$, the permutation $z^{m_2}$ acts as a single $m_1$-cycle, and conjugation by $\s_1$ simply increases
each entry by $1$. Similarly, $z^{m_1}$ acts as a product of two disjoint $m_2$-cycles, and again conjugation by
$\s_1$ increases each entry by $1$. So we find that
\[ w = \begin{cases}
(6, 14, \ldots, 4t-2, 10, 18, \ldots, 4t-6) & \textrm{if $t$ is even,} \\
(6, 14, \ldots, 4t-6)(10, 18, \ldots, 4t-2) & \textrm{if $t$ is odd.}
\end{cases} \]
Since $n \geq 13$ and $t^2 \equiv -1$ (mod $n$), it must be the case that $4 \leq t \leq n-4$.
Therefore, the permutation $w$ moves vertex $6$ while fixing vertices $4t+2$, $4t+6$, and $8t+2$. It follows that
$w^{-1} z^{m_1 m_2} w = (14, 4t+2)(4t+6, 8t+2)$. Hence $z^{m_1 m_2} w^{-1} z^{m_1 m_2} w = (6, 14, 4t+2)$,
giving us our desired $3$-cycle. So $\Lambda_1 \cong \Lambda_2 \cong A_n$, and the claim follows.
\end{proof}

We note that if we take $\calK = \{4, 4\}_{(1, 2)}$, then the analysis is similar, but $\G$ does not act independently
on class $1$ and class $2$. A calculation with GAP \cite{gap} shows that in this case, $\Lambda = A_5$
(and so $\G(\calP) \equiv A_5 \rtimes D_4$).

\begin{corollary}
The index of $X(\calP)$ in $\G$ is at most $8$.
\end{corollary}

\begin{proof}
Recall that if $\G(\calP)$ has presentation $\langle \s_1, \ldots, \s_{d} \mid \calR \rangle$, then
\[ X(\calP) = \langle \ch{\calR} \rangle^{\G(\calP)}. \]
If we can show that $x \in X(\calP)$, then it will follow that $\tau x \tau \in X(\calP)$, since $X(\calP)$ is normal.
Then it follows that $X(\calP)$ contains $\Lambda$, which has index $8$ in $\G$, and the claim will follow.

We noted earlier that $x = \s_2^{-1} \s_1$ acts as a translation on the vertices of a given class.
Similarly, defining $y = \s_1 \s_2^{-1} = \s_1 x \s_1^{-1}$, we see that this also acts as a translation
on the vertices of a given class.
Now, one of the defining relations for $\G(\calK)$ is $(\s_1^{-1} \s_2)^b(\s_1 \s_2^{-1})^c = \eps$ (see \cite{chiregas2}). In other words,
$x^{-b} y^c = \eps$. This can assumed to be one of the defining relations for $\G(\calP)$. The enantiomorphic form
of this word is $((\s_1 \s_2^2)^{-1} \s_2^{-1})^b ((\s_1 \s_2^2) \s_2)^c$. Using the relations that
$\s_2^4 = \eps$ and $\s_2 \s_1 = \s_1^{-1} \s_2^{-1}$, we can rewrite this word as
\[ (\s_2^{-2} \s_1^{-1} \s_2^{-1})^b (\s_1 \s_2^3)^c = (\s_2^{-1} \s_1)^b (\s_1 \s_2^{-1})^c, \]
which is $x^b y^c$. So $x^b y^c$ is an element of $X(\calP)$.
Furthermore, it is a nontrivial translation; otherwise, $\calK$ would not be chiral.
Now, since $b^2 + c^2$ is prime, any nontrivial translation generates the whole translation subgroup.
In particular, the group generated by $x^b y^c$ contains $x$. Therefore, $X(\calP)$ contains $x$
and the result follows.
\end{proof}

In this section we assumed that $n = b^2 + c^2$ is prime. The graph constructed from $\{4,4\}_{(b,c)}$ when $n$ is not prime is still a GPR graph of a chiral $4$-polytope with Schl\"afli type $\{4,4,2q\}$ for some $q$. However, the value of $q$ will depend on $b$ and $c$, and in general will not coincide with $2(n-2)$. The determination of the automorphism group of these $4$-polytopes requires an analysis different from the one developed above, particularly in the case when the translation to the right does not generate the entire translation subgroup of $\K$. However, the family $\mathbb{F}$ of toroidal maps $\{4,4\}_{(b,c)}$ with $n$ prime is of great interest; it contains infinitely many elements, and every other chiral toroidal map $\{4, 4\}_{(b,c)}$ is
the mix of regular toroidal maps with one or more maps in $\mathbb{F}$ (see \cite{Gabe} and \cite{MixMon} for the definition and results about mix of regular and chiral polytopes).

We conclude by noting that other possible definitions of $\tau$, matching the vertices of $G$ with the vertices of $G'$, may yield $4$-polytopes with smaller automorphism groups $\Gamma$. However, the determination of $\Gamma$ in such a general setting involves bigger complications and is out of the scope of this paper.


\begin{thebibliography}{99}
\bibitem{chiral-mix}
\newblock A. Breda, G. Jones, E. Schulte,
\newblock Constructions of chiral polytopes of small rank,
\newblock {\em Canad. J. Math.} {\bf 63} (2011), 1254--1283.


\bibitem{CoM}
   H. S. M. Coxeter, W. O. J. Moser,
{\em Generators and Relations for Discrete Groups},
4th edition, Springer-Verlag, Berlin-New York, 1980.

\bibitem{Gabe}
\newblock G. Cunningham,
\newblock Mixing chiral polytopes,
\newblock {\em J. Alg. Comb.} {\bf 36} (2012), 263--277.

\bibitem{dan}
   L. Danzer, Regular incidence-complexes and dimensionally unbounded sequences of such, I,
\emph{Convexity and Graph Theory (Jerusalem 1981)}, North-Holland Math. Stud. 87,
North-Holland (Amsterdam, 1984), 115--127.

\bibitem{gap}
\newblock The GAP Group,
\newblock GAP -- Groups, Algorithms, and Programming,
\newblock Version 4.4.10 (2007), (http://www.gap-system.org).

\bibitem{hartley}
\newblock M. I. Hartley,
\newblock All polytopes are quotients, and isomorphic polytopes are quotients by conjugate
subgroups,
\newblock {\em Discrete Comput. Geom.} {\bf 21} (1999), 289--298.


\bibitem{selfinvariance}
\newblock I. Hubard, A. Orbani\'c, A. I. Weiss,
\newblock Monodromy groups and self-invariance,
\newblock {\em Canadian Journal of Math} {\bf 61} (2009), 1300--1324.

\bibitem{ARP}
\newblock P. McMullen, E. Schulte,
\newblock {\em Abstract regular polytopes},
\newblock {\em Encyclopedia of Math. And its Applic.} {\bf 92}, Cambridge, 2002.

\bibitem{MixMon}
\newblock B. Monson, D. Pellicer, G. Williams,
\newblock Mixing and monodromy of abstract polytopes,
\newblock {\em Trans. Amer. Math. Soc.} {\bf 366} (2014), 2651--2681.

\bibitem{pre}
\newblock D. Pellicer,
\newblock CPR graphs and regular polytopes,
\newblock {\em European Journal of Combinatorics} {\bf 29} (2008), 59--71.

\bibitem{pre2}
\newblock D. Pellicer,
\newblock Extensions of regular polytopes with preassigned Schl\"afli symbol,
\newblock {\em Journal of Combinatorial Theory, Series A} {\bf 116} (2009), 303--313.

\bibitem{higherranks}
\newblock D. Pellicer,
\newblock A construction of higher rank chiral polytopes,
\newblock {\em Discrete Math.} {\bf 310} (2010), 1222–-1237.

\bibitem{chiral-problems}
\newblock D. Pellicer,
\newblock Developments and open problems on chiral polytopes,
\newblock {\em Ars Mathematica Contemporanea} {\bf 5} (2) (2012), 333--354.

\bibitem{pewe}
\newblock D. Pellicer, A. I. Weiss,
\newblock Generalized CPR graphs and applications,
\newblock {\em Contrib. Discrete Math.} {\bf 5} (2010), 76--105.

\bibitem{arr}
   E. Schulte, On arranging regular incidence-complexes as faces of higher-dimensional ones,
{\em European Journal of Combinatorics} {\bf 4} (1983), 375--384.

\bibitem{chiregas}
   E. Schulte and A. I. Weiss, Chiral polytopes, {\em Applied Geometry and Discrete
Mathematics (The ``Victor Klee Festschrift'')}, DIMACS Ser. Discrete Math. Theoret.
Comput. Sci. {\bf 4}, (1991), (eds. P. Gritzmann and B. Sturmfels), 493--516.

\bibitem{chiregas2}
   E. Schulte and A. I. Weiss, Chirality and projective linear groups, {\em Discrete Math.},
{\bf 131}, (1994), 221--261.

\bibitem{chiregas3}
   E. Schulte and A. I. Weiss, Free extensions of chiral polytopes,
{\em Canadian Journal of Mathematics}, {\bf 47} (3), (1995), 641--654.

\bibitem{white}
   A. White, {\em Graphs, groups and surfaces}, North-Holland Mathematics Studies,
Amsterdam, 1973.

\bibitem{Wielandt}
\newblock H. Wielandt,
\newblock {\em Finite permutation groups},
\newblock Translated from the German by R. Bercov Academic Press, New York-London, 1964.


\end{thebibliography}
\end{document}